\documentclass[11pt]{amsart}

\usepackage[article]{robertpreamble}
\usepackage[margin=1in]{geometry}

\begin{document}


\title{Unbounded Topology of Nodal Sets of Harmonic Functions}
\author{Robert Koirala}
\address{Department of Mathematics, University of California San Diego}
\email{rkoirala@ucsd.edu}
\date{\today}


\begin{abstract}
    For every integer \(n\ge 3\), every \(1\le \ell\le n-2\), and every sufficiently large integer \(m\), we construct harmonic functions \(u_{m,\ell}\) on the unit ball \(B_1(0)\subset\R^n\) such that the frequency is bounded independently of \(m\), every point of the nodal set \(\{u_{m,\ell}=0\}\cap B_{1/2}(0)\) is regular, but the Betti numbers satisfy
    \begin{align*}
        b_\ell\bigl(\{u_{m,\ell}=0\}\cap B_{1/2}(0)\bigr)\ge 2m.
    \end{align*}
    Thus bounded frequency, even together with regularity of the nodal set, does not imply a uniform topological bound. In particular, these examples give counterexamples to the claimed global Betti-number bound of Lin and Liu.
\end{abstract}

\maketitle

\section{Introduction}

Let \(u\) be a nonzero harmonic function on \(B_1(0)\subset\R^n\). For \(0<r<1\), define the Almgren frequency at scale \(r\) by
\begin{equation}\label{eq:frequency-def}
    \mathcal N_r(u)
    \coloneqq 
    \frac{r\int_{B_r(0)}|\nabla u|^2}{\int_{\partial B_r(0)}u^2}.
\end{equation}
This normalization has the property that if \(p\) is a homogeneous harmonic polynomial of degree \(d\), then \(\mathcal N_r(p)=d\) for every \(r>0\).

The guiding question is the following.

\begin{question}\label{question-main}
    Given a bound for the frequency of a harmonic function \(u\), what geometric, measure-theoretic, and topological properties of its nodal set \(u^{-1}(0)\) can be controlled?
\end{question}

For harmonic polynomials, the answer is closely tied to real algebraic geometry. If \(p\) is a polynomial of degree \(d\), then \(p^{-1}(0)\) is a real algebraic variety, and many quantitative properties of \(p^{-1}(0)\) can be bounded in terms of \(d\). Wongkew proved a degree-dependent upper bound on the volume of tubular neighborhoods of real algebraic varieties \cite{MR1211391}. Classical results of Milnor and Thom give degree-dependent bounds for the total Betti number of real algebraic varieties \cite{MR161339,MR200942}.

General harmonic functions share several local features with harmonic polynomials. Their nodal sets are real analytic varieties, and at small scales they can be approximated by homogeneous harmonic polynomials. Moreover, frequency bounds play the role of quantitative degree bounds in many geometric and measure-theoretic estimates for nodal sets. For example, frequency bounds imply quantitative control on tubular neighborhoods of nodal sets; see, for instance, \cite{logunov-2018-nodal,logunov-2018-nodal-sets, logunov-malinnikova-2018-nodal-sets,logunov-malinnikova-2020-review, logunov-malinnikova-2021-the-sharp, MR3688031} and the references therein.

The purpose of this note is to show that this analogy breaks down at the level of topology. Although a frequency bound controls many quantitative features of a harmonic function, it does not control the Betti numbers of its nodal set. Throughout the paper, Betti numbers are taken with coefficients in a fixed field \(\mathbb F\).

\begin{theorem}\label{thm:main-counterexample}
    Fix \(n\ge 3\) and \(1\le \ell\le n-2\). For every sufficiently large integer \(m\), there exists a harmonic function \(u_{m,\ell}\) on \(B_1(0)\subset\R^n\) such that:
    \begin{enumerate}[label=\textup{(\roman*)}]
        \item \(\mathcal N_1(u_{m,\ell})\) is bounded above independently of \(m\);
        \item every point of \(\{u_{m,\ell}=0\}\cap B_{1/2}(0)\) is regular, so the
        pointwise vanishing order at every such nodal point is exactly \(1\);
        \item
        \begin{align*}
            b_\ell\bigl(\{u_{m,\ell}=0\}\cap B_{1/2}(0)\bigr)\ge 2m.
        \end{align*}
    \end{enumerate}
\end{theorem}

This gives counterexamples to the claimed global Betti-number bound in \cite{MR3479524}, even in the real-analytic harmonic setting and even when all nodal points under consideration are regular. The issue is not local regularity. Rather, the obstruction is global: a local topological bound on each ball \(B_{r(y)}(y)\) does not imply a global topological bound unless the number of balls needed to cover the nodal set is also uniformly controlled. Compactness gives a finite subcover, but it gives no uniform bound on the cardinality of such a subcover. In our examples, the relevant quantitative scale degenerates near an almost-critical set, allowing arbitrarily many independent cycles while the frequency and the pointwise vanishing order remain uniformly bounded.

\begin{figure}[!htb]
    \centering
    \includegraphics[width=0.7\linewidth]{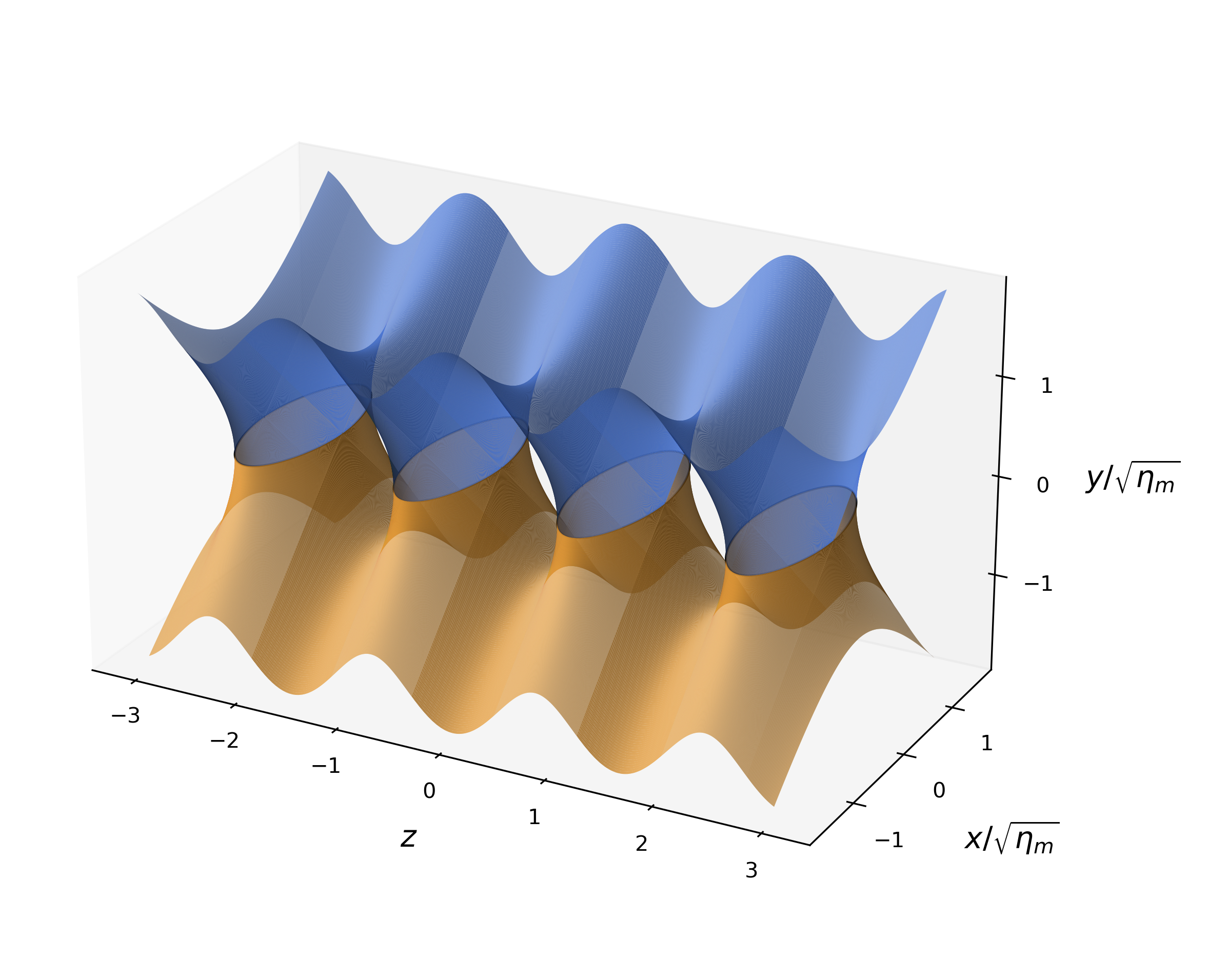}
    \caption{The nodal surface in the case \(\ell=1\), shown in rescaled coordinates for \(m=4\).}
    \label{fig:nodal-surface-m4}
\end{figure}

We briefly describe the construction. In the three-dimensional case, one starts from the harmonic polynomial \(x^2-y^2\), whose nodal set has a crossing along the \(z\)-axis, and perturbs it by an exponentially small oscillatory harmonic term:
\begin{align*}
    x^2-y^2+\eta_m e^{\lambda_m x}\cos(\lambda_m z),
    \qquad
    \lambda_m=8\pi m,\qquad
    \eta_m=e^{-\lambda_m}\lambda_m^{-4}.
\end{align*}
The amplitude is chosen so small that the frequency at scale \(1\) remains uniformly bounded and no singular nodal points are created.  At the same time, the oscillation in the \(z\)-variable creates \(2m\) small holes in the planar projection of the nodal surface, and loops around these holes lift to independent one-cycles.

For higher Betti numbers, we replace the crossing \(x^2-y^2\) by the harmonic quadratic cone
\begin{align*}
    |X|^2-\ell y^2,
    \qquad X\in\R^\ell.
\end{align*}
The same oscillatory perturbation creates \(2m\) holes in the \((X,z)\)-projection. Their boundaries are \(\ell\)-spheres, and these spheres lie in the nodal set. Projection to punctured \(\R^{\ell+1}\) then shows that the corresponding homology classes are independent.

\subsection*{Acknowledgments} We thank Bennett Chow for encouragement.

\section{Proof of the main theorem}

Throughout the proof, \(C\) denotes a positive constant independent of \(m\), though it may depend on \(n\) and \(\ell\). Fix \(n\ge3\) and \(1\le \ell\le n-2\). Write
\begin{align}
    (X,y,z,w)\in \R^\ell\times\R\times\R\times\R^{n-\ell-2},
    \qquad X=(x_1,\ldots,x_\ell),
\end{align}
with the convention that the \(w\)-variables are absent when \(\ell=n-2\). Set
\begin{equation}\label{eq:lambda-eta}
    \lambda_m\coloneqq 8\pi m,
    \qquad
    \eta_m\coloneqq e^{-\lambda_m}\lambda_m^{-4},
\end{equation}
and define
\begin{equation}\label{eq:u-ell-def}
    u_{m,\ell}(X,y,z,w)
    \coloneqq 
    |X|^2-\ell y^2+
    \eta_m e^{\lambda_m x_1}\cos(\lambda_m z).
\end{equation}

\subsection{Frequency and regularity}

It is straight forward to check that \(u_{m,\ell}\) satisfies \(\Delta u_{m,\ell}=0\) in \(B_1(0)\subset\R^n\). Moreover,
\begin{equation}\label{eq:C1-small-general}
    \left\|\eta_m e^{\lambda_m x_1}\cos(\lambda_m z)\right\|_{L^\infty(B_1)}
    \le \lambda_m^{-4},
    \qquad
    \left\|\nabla\bigl(\eta_m e^{\lambda_m x_1}\cos(\lambda_m z)\bigr)\right\|_{L^\infty(B_1)}
    \le C\lambda_m^{-3}.
\end{equation}
Hence \(u_{m,\ell}\longrightarrow Q_\ell\coloneqq |X|^2-\ell y^2\) in \(C^1(\overline{B_1}).\) It follows that
\begin{align*}
    \int_{\partial B_1} u_{m,\ell}^2\,d\sigma
    \longrightarrow
    \int_{\partial B_1} Q_\ell^2\,d\sigma>0, \qquad
    \int_{B_1}|\nabla u_{m,\ell}|^2\,dx
    \longrightarrow
    \int_{B_1}|\nabla Q_\ell|^2\,dx.
\end{align*}
Therefore, for all sufficiently large \(m\),
\begin{equation}\label{eq:frequency-bound-general}
    \mathcal N_1(u_{m,\ell})\le C.
\end{equation}

\begin{lemma}\label{lem:regular-zero-general}
    For all sufficiently large \(m\), zero is a regular value of \(u_{m,\ell}\) in \(B_1(0)\). In particular, every nodal point of \(u_{m,\ell}\) in \(B_{1/2}(0)\) has pointwise vanishing order exactly \(1\).
\end{lemma}

\begin{proof}
    We compute
    \begin{align*}
        \nabla u_{m,\ell}
        =
        \bigl(
            2x_1+
            \eta_m\lambda_m e^{\lambda_m x_1}\cos(\lambda_m z),
            2x_2,\ldots,2x_\ell,
            -2\ell y,
            -\eta_m\lambda_m e^{\lambda_m x_1}\sin(\lambda_m z),
            0_w
        \bigr).
    \end{align*}
    Suppose, for the sake of contradiction, that \(u_{m,\ell}=0\) and \(\nabla u_{m,\ell}=0\) at some point of \(B_1(0)\). Then
    \begin{align*}
        x_2=\cdots=x_\ell=0,
        \qquad
        y=0,
        \qquad
        \sin(\lambda_m z)=0.
    \end{align*}
    Thus \(\cos(\lambda_m z)=\pm1\). If \(\cos(\lambda_m z)=1\), then
    \begin{align*}
        u_{m,\ell}=x_1^2+
        \eta_m e^{\lambda_m x_1}>0,
    \end{align*}
    contradicting \(u_{m,\ell}=0\). Hence \(\cos(\lambda_m z)=-1\). The equations \(u_{m,\ell}=0\) and \(\partial_{x_1}u_{m,\ell}=0\) become
    \begin{equation}\label{eq:critical-system-general}
        x_1^2=\eta_m e^{\lambda_m x_1},
        \qquad
        2x_1=\eta_m\lambda_m e^{\lambda_m x_1}.
    \end{equation}
    The first equation gives \(x_1\ne0\). Dividing the second equation by the first gives \(x_1=\frac{2}{\lambda_m}.\) Substitution into the first equation in \eqref{eq:critical-system-general} gives
    \begin{align*}
        \frac{4}{\lambda_m^2}
        =
        \eta_m e^2
        =
        e^{-\lambda_m}e^2\lambda_m^{-4},
    \end{align*}
    which is impossible for all sufficiently large \(m\). Therefore \(\nabla u_{m,\ell}\ne0\) at every point of \(\{u_{m,\ell}=0\}\cap B_1(0)\).
\end{proof}

\subsection{Projected holes}

Define
\begin{equation}\label{eq:Phi-general-def}
    \Phi_{m,\ell}(X,z)
    \coloneqq 
    |X|^2+
    \eta_m e^{\lambda_m x_1}\cos(\lambda_m z).
\end{equation}
Then \(\{u_{m,\ell}=0\} = \{(X,y,z,w):\ell y^2=\Phi_{m,\ell}(X,z)\}.\) Thus, after suppressing the \(w\)-variables, the nodal set is a two-sheeted graph over the region \(\{\Phi_{m,\ell}\ge0\}\), with sheets \(y=\pm\sqrt{\Phi_{m,\ell}(X,z)/\ell}\) meeting along the branch set \(\{\Phi_{m,\ell}=0,\ y=0\}\).

Let
\begin{equation}\label{eq:R-ell-def}
    \sigma\coloneqq \frac1{16},
    \qquad
    R_\ell\coloneqq \overline{B^\ell_\sigma(0)}\times
    \left[-\frac14,\frac14\right]
    \subset \R^{\ell+1}_{X,z}.
\end{equation}
Since \(\lambda_m=8\pi m\), the phase \(\lambda_m z\) runs from \(-2\pi m\) to \(2\pi m\) as \(z\) runs over \([-1/4,1/4]\). The intervals on which \(\cos(\lambda_m z)<0\) in this window are
\begin{equation}\label{eq:intervals-general}
    I_{j,m}
    \coloneqq 
    \left(
        \frac{\pi/2+2\pi j}{\lambda_m},
        \frac{3\pi/2+2\pi j}{\lambda_m}
    \right),
    \qquad
    j=-m,-m+1,\ldots,m-1.
\end{equation}
There are exactly \(2m\) such intervals.

For \(j=-m,\ldots,m-1\), define
\begin{equation}\label{eq:Hjm-general-def}
    H_{j,m}^{(\ell)}
    \coloneqq 
    \{(X,z)\in R_\ell:z\in I_{j,m},\ \Phi_{m,\ell}(X,z)<0\}.
\end{equation}

\begin{lemma}\label{lem:holes-general}
    For all sufficiently large \(m\), each \(H_{j,m}^{(\ell)}\) is an open topological \((\ell+1)\)-ball whose closure is compactly contained in \(R_\ell\). The \(2m\) sets \(H_{j,m}^{(\ell)}\) are pairwise disjoint, and
    \begin{align}
        D_{m,\ell}\coloneqq R_\ell\cap\{\Phi_{m,\ell}\ge0\}
    \end{align}
    is \(R_\ell\) with these \(2m\) open balls removed. In particular, \(\partial H_{j,m}^{(\ell)}\) is a smooth \(\ell\)-sphere.
\end{lemma}

\begin{figure}[!htb]
    \centering
    \includegraphics[width=\linewidth]{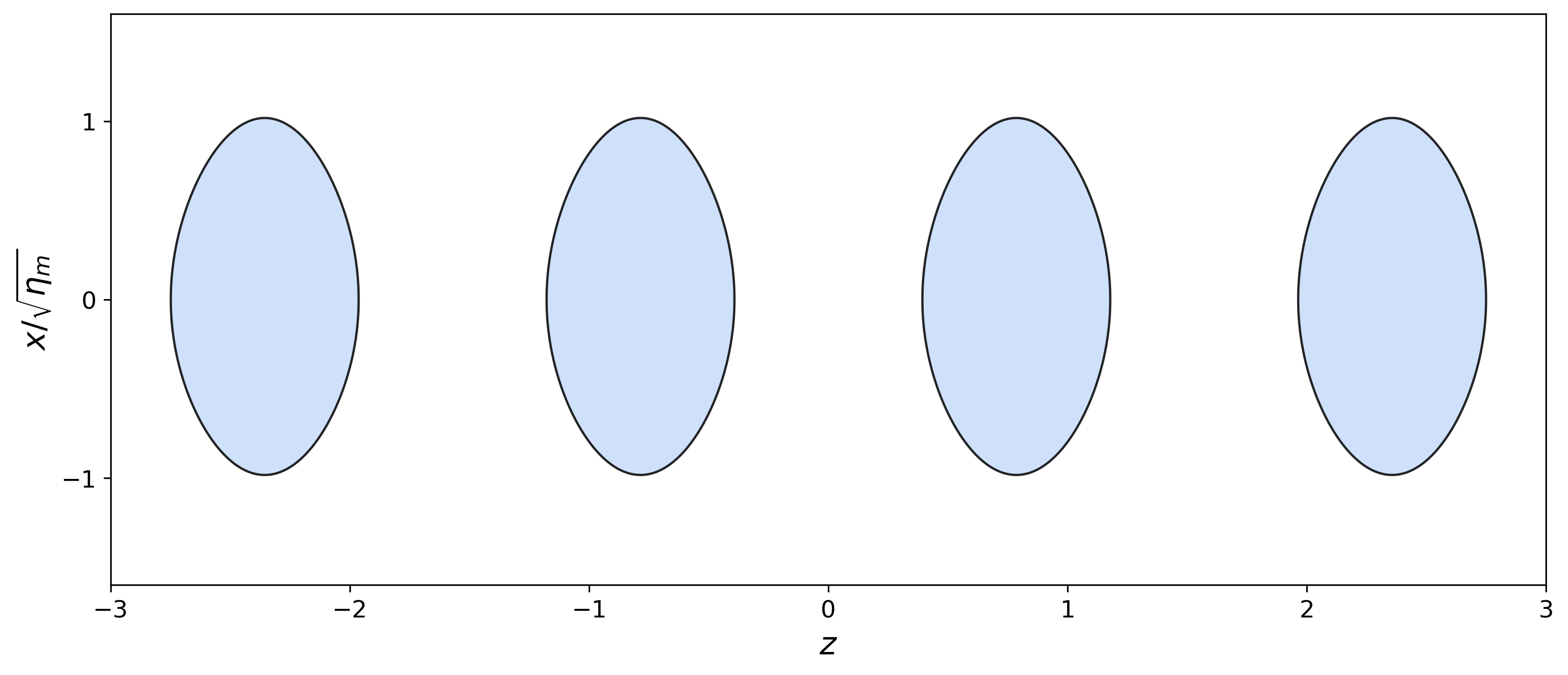}
    \caption{Projected disks in the \((x,z)\)-plane for \(\ell=1\), \(m=4\), \(n=3\).}
    \label{fig:projected-holes-m4}
\end{figure}

\begin{proof}
    We first show that \(\Phi_{m,\ell}>0\) on \(\partial R_\ell\) for all sufficiently large \(m\). On the horizontal sides \(z=\pm1/4\), we have \(\cos(\lambda_m z)=\cos(\pm 2\pi m)=1\), and hence
    \begin{align*}
        \Phi_{m,\ell}(X,\pm1/4)
        =
        |X|^2+
        \eta_m e^{\lambda_m x_1}>0.
    \end{align*}
    On the vertical side \(|X|=\sigma\), using \(x_1\le |X|=\sigma\), we get
    \begin{align*}
        \Phi_{m,\ell}(X,z)
        \ge
        \sigma^2-
        \eta_m e^{\lambda_m\sigma}
        =
        \sigma^2-e^{-(1-\sigma)\lambda_m}\lambda_m^{-4}>0
    \end{align*}
    for all sufficiently large \(m\).
    
    Next, on \(R_\ell\),
    \begin{align*}
        D_X^2\Phi_{m,\ell}(X,z)
        =
        2I_\ell+
        \eta_m\lambda_m^2e^{\lambda_m x_1}\cos(\lambda_m z)
        \, e_1\otimes e_1
        \ge I_\ell
    \end{align*}
    for all sufficiently large \(m\). Hence, for each fixed \(z\), the function \(X\mapsto \Phi_{m,\ell}(X,z)\) is strictly convex on \(\overline{B^\ell_\sigma(0)}\).
    
    If \(\cos(\lambda_m z)\ge0\), then
    \begin{align*}
        \Phi_{m,\ell}(X,z)\ge |X|^2\ge0,
    \end{align*}
    so there is no negative set over that value of \(z\). If \(z\in I_{j,m}\), then
    \begin{align*}
        \Phi_{m,\ell}(0,z)=\eta_m\cos(\lambda_m z)<0,
    \end{align*}
    whereas \(\Phi_{m,\ell}>0\) on \(|X|=\sigma\). Therefore the slice \(\{X\in B^\ell_\sigma(0):\Phi_{m,\ell}(X,z)<0\}\) is a nonempty bounded strictly convex domain. Its boundary is smooth: a critical point of the strictly convex function \(X\mapsto \Phi_{m,\ell}(X,z)\) is its unique minimum, while \(\Phi_{m,\ell}(0,z)<0\), so the zero level cannot contain a critical point in the \(X\)-variables.

    It remains to understand the topology of \(H_{j,m}^{(\ell)}\). Let \(z_-\) and \(z_+\) be the endpoints of \(I_{j,m}\). At either endpoint \(z_0\), one has \(\cos(\lambda_m z_0)=0\), and hence
    \begin{align*}
        \Phi_{m,\ell}(X,z_0)=|X|^2.
    \end{align*}
    Thus the negative slice collapses to the single point \(X=0\). Moreover,
    \begin{align*}
        \partial_z\Phi_{m,\ell}(0,z_0)
        =
        -\eta_m\lambda_m\sin(\lambda_m z_0)\ne0,
    \end{align*} 
    so \(\{\Phi_{m,\ell}=0\}\) is a smooth hypersurface near \((0,z_0)\).
    
    For each \(z\in I_{j,m}\), the strictly convex function \(X\mapsto \Phi_{m,\ell}(X,z)\) has a unique minimizer \(c(z)\), and \(c(z)\) depends smoothly on \(z\). Since the negative slice collapses to \(\{0\}\) at the endpoints, \(c(z)\to0\) as \(z\to z_\pm\). For each \(\omega\in S^{\ell-1}\), strict convexity gives a unique number \(\rho(z,\omega)>0\) such that
    \begin{align*}
        \Phi_{m,\ell}\bigl(c(z)+\rho(z,\omega)\omega,z\bigr)=0.
    \end{align*}
    Moreover \(\rho(z,\omega)\to0\) as \(z\to z_\pm\), uniformly in \(\omega\). Therefore
    \begin{align*}
        (z,t,\omega)
        \longmapsto
        \bigl(c(z)+t\rho(z,\omega)\omega,z\bigr),
        \qquad
        0\le t\le1,
    \end{align*}
    parametrizes \(\overline{H_{j,m}^{(\ell)}}\), with the two endpoint fibers collapsed to points. The displayed parametrization induces a homeomorphism from a closed \((\ell+1)\)-ball onto \(\overline{H_{j,m}^{(\ell)}}\). Its boundary is the smooth hypersurface \(\{\Phi_{m,\ell}=0\}\cap \bigl(\overline{B^\ell_\sigma(0)}\times \overline{I_{j,m}}\bigr),\) and hence is a smooth \(\ell\)-sphere.
    
    The intervals \(I_{j,m}\) are pairwise disjoint, so the corresponding balls are pairwise disjoint. Since \(\Phi_{m,\ell}>0\) on \(\partial R_\ell\), their closures are compactly contained in \(R_\ell\). Finally, the negative set of \(\Phi_{m,\ell}\) in \(R_\ell\) is exactly the union of the \(H_{j,m}^{(\ell)}\), which proves the lemma.
\end{proof}

\subsection{Independent higher-dimensional cycles}

Let
\begin{equation}\label{eq:Z-ell-def}
    Z_{m,\ell}\coloneqq \{u_{m,\ell}=0\}\cap B_{1/2}(0).
\end{equation}

\begin{lemma}\label{lem:higher-betti-lower}
    For all sufficiently large \(m\),
    \begin{align*}
        b_\ell(Z_{m,\ell})\ge 2m.
    \end{align*}
\end{lemma}

\begin{proof}
    Choose one point
    \begin{align*}
        p_{j,m}\in H_{j,m}^{(\ell)},
        \qquad j=-m,\ldots,m-1,
    \end{align*}
    and set
    \begin{align*}
        P_m\coloneqq \{p_{j,m}:j=-m,\ldots,m-1\}\subset\R^{\ell+1}_{X,z}.
    \end{align*}
    Since \(\Phi_{m,\ell}(p_{j,m})<0\), no point of the nodal set projects to \(p_{j,m}\). Indeed, if \((X,z)=p_{j,m}\), then
    \begin{align*}
        u_{m,\ell}(X,y,z,w)
        =
        \Phi_{m,\ell}(p_{j,m})-\ell y^2<0.
    \end{align*}
    Therefore the projection
    \begin{equation}\label{eq:projection-general}
        \pi:Z_{m,\ell}\to \R^{\ell+1}\setminus P_m,
        \qquad
        \pi(X,y,z,w)=(X,z),
    \end{equation}
    is well-defined.
    
    For each \(j\), define
    \begin{align*}
        \Gamma_{j,m}
        \coloneqq 
        \{(X,0,z,0):(X,z)\in\partial H_{j,m}^{(\ell)}\}.
    \end{align*}
    Since \(\Phi_{m,\ell}=0\) on \(\partial H_{j,m}^{(\ell)}\), we have \(\Gamma_{j,m}\subset Z_{m,\ell}\). Moreover, \(\Gamma_{j,m}\cong \partial H_{j,m}^{(\ell)}\cong S^\ell.\) These spheres lie in \(B_{1/2}(0)\). Indeed, if \((X,z)\in R_\ell\), then
    \begin{align*}
        |X|\le \sigma=\frac1{16},
        \qquad
        |z|\le\frac14,
    \end{align*}
    and on \(\Gamma_{j,m}\) one has \(y=0\) and \(w=0\). Thus
    \begin{align*}
        |X|^2+y^2+z^2+|w|^2
        \le
        \frac1{256}+\frac1{16}
        <
        \frac14.
    \end{align*}

    The homology group \(H_\ell(\R^{\ell+1}\setminus P_m;\mathbb F)\) is naturally isomorphic to \(\mathbb F^{2m}\), with basis represented by small positively oriented \(\ell\)-spheres surrounding the points of \(P_m\). Since \(\partial H_{j,m}^{(\ell)}\) surrounds exactly \(p_{j,m}\), the classes \([\partial H_{j,m}^{(\ell)}],\) \(j=-m,\ldots,m-1,\) form this standard basis. By construction, \(\pi_*[\Gamma_{j,m}] = [\partial H_{j,m}^{(\ell)}].\) Thus
    \begin{align*}
        \pi_*:H_\ell(Z_{m,\ell};\mathbb F)
        \longrightarrow
        H_\ell(\R^{\ell+1}\setminus P_m;\mathbb F)
    \end{align*}
    is surjective.  Therefore
    \begin{align*}
        b_\ell(Z_{m,\ell})\ge 2m.
    \end{align*}
    This proves the lemma.
\end{proof}

\begin{proof}[Proof of Theorem \ref{thm:main-counterexample}]
    The harmonic functions \(u_{m,\ell}\) are defined in \eqref{eq:u-ell-def}. The uniform frequency bound follows from \eqref{eq:frequency-bound-general}, the regularity assertion follows from Lemma \ref{lem:regular-zero-general}, and the Betti-number lower bound follows from Lemma \ref{lem:higher-betti-lower}. This completes the proof.
\end{proof}

\begin{remark}\label{rem:regularized-betti}
    The same examples also give unbounded topology for the regularized Betti-number definition used in \cite{MR3479524}. Fix \(m\) and \(\ell\), and choose the points \(p_{j,m}\in H_{j,m}^{(\ell)}\).  Since there are only finitely many of them,
    \begin{align*}
        a_m\coloneqq \min_{j=-m,\ldots,m-1}\bigl(-\Phi_{m,\ell}(p_{j,m})\bigr)>0.
    \end{align*}
    Let
    \begin{align*}
        \Sigma_{\varepsilon,\theta}
        \coloneqq 
        \{u_{m,\ell}^2+\varepsilon^2|V|^2=\theta^2\}\cap B_{1/2}(0),
        \qquad V=(X,y,z,w).
    \end{align*}
    If \(0<\theta<a_m/2\), then no point of \(\Sigma_{\varepsilon,\theta}\) projects to \(P_m\), because over \(p_{j,m}\) one has
    \begin{align*}
        u_{m,\ell}(X,y,z,w)
        =
        \Phi_{m,\ell}(p_{j,m})-\ell y^2
        \le -a_m.
    \end{align*}
    Thus the projection to \(\R^{\ell+1}\setminus P_m\) is well-defined on \(\Sigma_{\varepsilon,\theta}\). The cycles \(\Gamma_{j,m}\) lie in a compact subset of the regular part of the nodal set. Hence \(|\nabla u_{m,\ell}|\) is bounded below on a small tubular neighborhood of their union. For all sufficiently small \(\theta>0\) and all \(0<\varepsilon\ll\theta\), the equation
    \begin{align*}
        u_{m,\ell}^2+\varepsilon^2|V|^2=\theta^2
    \end{align*}
    can therefore be solved in this tubular neighborhood as two nearby normal graphs over the nodal set. In particular, each \(\Gamma_{j,m}\) has a nearby cycle \(\Gamma_{j,m}^{\varepsilon,\theta}\subset \Sigma_{\varepsilon,\theta}.\) Their projections are small perturbations of \(\partial H_{j,m}^{(\ell)}\), and therefore represent the same standard generators of \(H_\ell(\R^{\ell+1}\setminus P_m;\mathbb F)\). Hence
    \begin{align*}
        b_\ell(\Sigma_{\varepsilon,\theta})\ge 2m
    \end{align*}
    for such \(\varepsilon\) and \(\theta\). Taking
    \((\varepsilon,\theta)\to(0,0)\) with \(\varepsilon/\theta\to0\) gives the same lower bound for the regularized limsup.
\end{remark}


\bibliographystyle{amsalpha}
\bibliography{references}

\end{document}